\documentclass[preprint,12pt]{elsarticle}

\usepackage{amsmath,amssymb,amsthm,mathtools,mathrsfs}
\usepackage{lmodern}
\usepackage[hidelinks]{hyperref}
\allowdisplaybreaks
\journal{Journal of Differential Equations}
\biboptions{sort&compress}

\numberwithin{equation}{section}
\newtheorem{theorem}{Theorem}[section]
\newtheorem{proposition}{Proposition}[section]
\newtheorem{lemma}{Lemma}[section]

\theoremstyle{definition}

\theoremstyle{remark}

\newcommand{\R}{\mathbb R}
\newcommand{\cL}{\mathcal L}
\newcommand{\Ha}{\dot H_a^1}
\newcommand{\Ham}{\dot H_a^{-1}}
\newcommand{\norm}[2]{\lVert #1\rVert_{#2}}

\newcommand{\eps}{\varepsilon}

\begin{document}

\begin{frontmatter}

\title{Strong and Weak-Type Dispersive Estimates for the Energy-Critical Nonlinear Schrödinger Equation with an Inverse-Square Potential}

\author[]{Tiesong Jiang}
\ead{1162163756jts@stu.xjtu.edu.cn}

\author[]{Kexue Li\corref{cor1}}
\ead{kxli@mail.xjtu.edu.cn}

%\author[xjtu]{Tiesong Jiang\corref{cor1}}
%\ead{}

%\author[xjtu]{Kexue Li}
%\ead{}

\cortext[cor1]{Corresponding author}

%\affiliation[xjtu]{
% organization={School of Mathematics, Xi'an Jiaotong University},
% city={Xi'an 710049},
% country={China}
%}

\begin{abstract}
We prove dispersive estimates for the three-dimensional defocusing energy-critical nonlinear Schr\"odinger equation associated with $\mathcal L_a=-\Delta+a|x|^{-2}$. For nonnegative potentials, we extend the known finite-$p$ theory to the endpoint $L^1\to L^\infty$. For negative potentials in the global well-posedness range, set $\sigma=\frac12-\sqrt{\frac14+a}$. We obtain the free strong decay rate for $2<p<3/\sigma$ and the limiting Lorentz estimate from $L^{(3/\sigma)',1}$ to $L^{3/\sigma,\infty}$. The restriction $p<3/\sigma$ is sharp for strong Lebesgue decay. The proof combines a finite-interval bootstrap, nonlinear real interpolation on bounded energy sets, and endpoint Sobolev--Lorentz estimates adapted to $\mathcal L_a$.
\end{abstract}


\end{frontmatter}

\section{Introduction}

We consider the defocusing energy-critical nonlinear Schr\"odinger equation
\begin{equation}
 \begin{cases}
  i\partial_tu-\cL_a u=|u|^4u, & (t,x)\in\R\times\R^3,\\
  u(0,x)=u_0(x),
 \end{cases}
 \qquad \cL_a=-\Delta+\dfrac{a}{|x|^2}.
 \label{eq:nls}
\end{equation}
The operator $\cL_a$ is the Friedrichs extension of the quadratic form
\[
 Q_a(f)=\int_{\R^3}\left(|\nabla f(x)|^2+
                  \frac{a}{|x|^2}|f(x)|^2\right)\,dx,
 \qquad f\in C_c^\infty(\R^3\setminus\{0\}).
\]
The sharp Hardy inequality implies that $Q_a$ is nonnegative for
$a\geq-\frac14$ as discussed in \cite{KillipSobolev}, for $a>-\frac14$,
\[
 \norm{\cL_a^{1/2}f}{L^2}\simeq_a\norm{\nabla f}{L^2}.
\]
Thus the form domain, denoted by $\Ha$, is the usual homogeneous energy
space $\dot H^1(\R^3)$ with an equivalent norm. Its dual with respect
to the distributional pairing is denoted by $\Ham$ and has norm
$\norm{F}{\Ham}=\norm{\cL_a^{-1/2}F}{L^2}$. We write
$U_a(t)=e^{-it\cL_a}$. Both the equation and the energy
\[
 E_a(u)=\frac12 Q_a(u)+\frac16\int_{\R^3}|u(x)|^6\,dx
\]
are invariant under $u(t,x)\mapsto\lambda^{1/2}u(\lambda^2t,\lambda x)$.

The global theory needed here is due to Killip, Miao, Visan, Zhang, and
Zheng \cite[Theorem~1.2]{KillipNLS} and continuous dependence follows from
\cite[Theorem~2.11]{KillipNLS}.

\begin{theorem}[Global well-posedness and scattering]
\label{thm:global}
Let $a>-\frac14+\frac1{25}$. For every $u_0\in\dot H^1(\R^3)$,
equation \eqref{eq:nls} has a unique global solution
\[
 u\in C(\R;\dot H^1(\R^3))\cap L^{10}_{t,x}(\R\times\R^3)
\]
satisfying
\begin{equation}
 \int_{\R}\int_{\R^3}|u(t,x)|^{10}\,dx\,dt
 \leq C_a\bigl(\norm{u_0}{\dot H^1}\bigr).
 \label{eq:critical-bound}
\end{equation}
Moreover, there are unique $u_\pm\in\dot H^1(\R^3)$ such that
\[
 \lim_{t\to\pm\infty}
 \norm{u(t)-U_a(t)u_\pm}{\dot H^1}=0.
\]
On bounded subsets of $\dot H^1$, the solution map is continuous in
$C([-T,T];\dot H^1)$ for every $T<\infty$.
\end{theorem}

For the free equation, the dispersive estimate
\[
 \norm{e^{it\Delta}f}{L^p}
 \lesssim |t|^{-3(\frac12-\frac1p)}\norm{f}{L^{p'}},
 \qquad 2\leq p\leq\infty,
\]
describes the spreading of a linear wave. Recovering this rate for a
critical nonlinear solution requires global spacetime information in
addition to scattering. Representative results for nonlinear
Schr\"odinger equations include
\cite{FanZhao2021,FanZhao2023,FanStaffilaniZhao,GuoHuangSong,
FanKillipVisanZhao}. Kowalski proved the full range $2<p\leq\infty$ for
the three-dimensional energy-critical equation without a potential
\cite{Kowalski}. For $a\geq0$, Wang, Xu, and Zhang established the
corresponding inverse-square estimate for every finite $p>2$
\cite[Theorem~1.2]{WangXuZhang}. Their final Sobolev step is not
available at $p=\infty$.

For $a\ge0$, we overcome this difficulty by using the endpoint Sobolev-Lorentz embedding $\dot W^{1,(3,1)}\hookrightarrow L^\infty$. Together with the adapted Riesz-transform bound and Lorentz interpolation of the linear dispersive estimate, this yields

\[
\|U_a(t)f\|_{L^\infty}
\lesssim_a |t|^{-1/2}
\|\mathcal L_a^{1/2}f\|_{L^{3/2,1}}.
\]

The corresponding nonlinear estimate and global Lorentz spacetime bounds then control the late Duhamel term, allowing the finite-interval bootstrap to close at $p=\infty$.

Attractive potentials have a different linear range. Put
\[
 \sigma=\frac12-\sqrt{\frac14+a}.
\]
If $-\frac14+\frac1{25}<a<0$, then $0<\sigma<\frac3{10}$. The
$\ell=0$ spherical-harmonic component of the kernel has the local
profile $|x|^{-\sigma}|y|^{-\sigma}$; see \cite[Eqs. (1.24), (1.27) and Section 6]{FanelliDecay}. The stationary wave-operator
theorem of Miao, Su, and Zheng \cite{MiaoSuZheng} consequently yields
strong unweighted linear decay only for $2\leq p<3/\sigma$. At
$p=3/\sigma$ the natural target is weak $L^p$.

The three decay statements proved in this article are collected below.
All control functions may be chosen nondecreasing in the displayed
energy norm.

\begin{theorem}
\label{thm:main}
Let $u$ be the global solution in Theorem~\ref{thm:global}.
\begin{enumerate}
 \item If $a\geq0$ and $u_0\in\dot H^1(\R^3)\cap L^1(\R^3)$, then
 \begin{equation}
  \norm{u(t)}{L^\infty}
  \leq C_a\bigl(\norm{u_0}{\dot H^1}\bigr)
  |t|^{-3/2}\norm{u_0}{L^1},
  \qquad t\neq0.
  \label{eq:nonnegative-endpoint}
 \end{equation}
 \item Suppose $-\frac14+\frac1{25}<a<0$ and define $\sigma$ as
 above.
 \begin{enumerate}
  \item If $2<p<3/\sigma$ and
  $u_0\in\dot H^1(\R^3)\cap L^{p'}(\R^3)$, then
  \begin{equation}
   \norm{u(t)}{L^p}
   \leq C_{a,p}\bigl(\norm{u_0}{\dot H^1}\bigr)
   |t|^{-3(\frac12-\frac1p)}\norm{u_0}{L^{p'}},
   \qquad t\neq0.
   \label{eq:attractive-strong}
  \end{equation}
  \item Let $p_*=3/\sigma$. If
  $u_0\in\dot H^1(\R^3)\cap L^{p_*',1}(\R^3)$, then
  \begin{equation}
   \norm{u(t)}{L^{p_*,\infty}}
   \leq C_{a,p_*}\bigl(\norm{u_0}{\dot H^1}\bigr)
   |t|^{-3(\frac12-\frac1{p_*})}
   \norm{u_0}{L^{p_*',1}},
   \qquad t\neq0.
   \label{eq:attractive-endpoint}
  \end{equation}
 \end{enumerate}
\end{enumerate}
\end{theorem}

Combining part (1) with \cite[Theorem~1.2]{WangXuZhang} gives the
strong decay estimate for $a\geq0$ and every $2<p\leq\infty$. For an
attractive potential the upper strong endpoint is sharp already for
the linear flow; see Proposition~\ref{prop:sharpness}. The Lorentz
estimate \eqref{eq:attractive-endpoint} is therefore the limiting
unweighted statement.

 The proof combines global Lorentz spacetime bounds with the interval bootstrap of \cite[Section~3]{Kowalski} and \cite[Section~4]{WangXuZhang}. For \(a<0\), we remove the lower restriction on \(p\) imposed by the adapted Sobolev estimates through Lipschitz decay from \(L^{7/5}\) to \(L^{7/2}\) and nonlinear real interpolation on bounded energy sets. For \(6\le p<3/\sigma\), the \(L^{7/2}\) decay controls the early Duhamel integral, while fractional Sobolev and product estimates control the late integral with time kernel \((t-s)^{-3/4}\). At \(p=3/\sigma\), a weak-type fractional integral estimate permits the same bootstrap in \(L^{p,\infty}\). For \(a\ge0\), the endpoint Sobolev-Lorentz embedding \(\dot W^{1,(3,1)}\hookrightarrow L^\infty\) supplies the missing Sobolev step and yields decay at \(p=\infty\).

The paper is organized as follows. Section 2 collects the linear dispersive and Strichartz estimates, adapted Sobolev inequalities, and global Lorentz spacetime bounds. Section 3 establishes the difference estimates used in nonlinear interpolation, the fractional and endpoint Sobolev-Lorentz estimates, and the early-time bounds and local integrability needed for the bootstrap. It also proves sharpness of the strong range for \(a<0\). Section 4 proves Theorem 1.2 by finite-interval iteration, treating the strong estimates and the two endpoints separately, and then extends the almost-everywhere decay bounds to every \(t\ne0\).

\section{Preliminaries}

We write $X\lesssim Y$ when $X\leq CY$ and when the constant $C$ depends on a parameter $b$, we write $X\lesssim_b Y$. The exponent $p'$ is the
H\"older conjugate of $p$. A pair $(q,r)$ is Schr\"odinger-admissible
in dimension three if
\[
 2\leq q,r\leq\infty,
 \qquad \frac2q+\frac3r=\frac32.
\]
The notation $L^{p,s}$ refers to a Lorentz space; $L^{p,p}=L^p$ with
equivalent norms, and $L^{p,s_1}\hookrightarrow L^{p,s_2}$ when
$s_1\leq s_2$. We utilize the notation $L^q_t L^r_x$ to denote the space-time norm.

We use Lorentz H\"older and Young--O'Neil convolution in their standard
forms
\[
 \norm{fg}{L^{r,s}}
 \lesssim\norm{f}{L^{r_1,s_1}}\norm{g}{L^{r_2,s_2}},
 \qquad
 \norm{f*g}{L^{r,s}}
 \lesssim\norm{f}{L^{r_1,s_1}}\norm{g}{L^{r_2,s_2}},
\]
whenever the primary indices satisfy, respectively,
$1/r=1/r_1+1/r_2$ and $1+1/r=1/r_1+1/r_2$, and the secondary indices
satisfy $1/s\leq1/s_1+1/s_2$; see
\cite{ONeil}. We also use the real-interpolation
identity
\[
 (L^{r_0},L^{r_1})_{\theta,s}=L^{r,s},
 \qquad \frac1r=\frac{1-\theta}{r_0}+\frac\theta{r_1},
\]
from \cite[Chapter~5]{BerghLofstrom}; the Lorentz form of operator
interpolation goes back to Hunt \cite{Hunt}. In one time dimension,
$|t|^{-1/q}\in L^{q,\infty}$ for $0<q<\infty$.

For $1<p<\infty$, the associate norm
\[
 \norm{f}{(p,\infty)}
 =\sup_{\norm{g}{L^{p',1}}\leq1}
   \left|\int_{\R^3}f(x)\overline{g(x)}\,dx\right|
\]
is equivalent to the usual weak-$L^p$ quasi-norm. Moreover,
$(L^{p',1})^*=L^{p,\infty}$ and $L^{p',1}$ is separable
\cite[Chapter~IV]{BennettSharpley}. Every weak endpoint norm in the
proof is understood in this associate-norm form, in particular in the
Volterra estimates and in the passage from almost every time to a
fixed time.

\subsection{Linear and Strichartz estimates}

\begin{lemma}[Linear decay]
\label{lem:linear}
Let $t\neq0$.
\begin{enumerate}
 \item If $a\geq0$, then, for $2\leq r\leq\infty$,
 \begin{equation}
  \norm{U_a(t)f}{L^r}
  \lesssim_{a,r}|t|^{-3(\frac12-\frac1r)}\norm{f}{L^{r'}}.
  \label{eq:linear-nonnegative}
 \end{equation}
 \item If $-\frac14<a<0$ and
 $\sigma=\frac12-\sqrt{\frac14+a}$, then
 \begin{equation}
  \norm{U_a(t)f}{L^r}
  \lesssim_{a,r}|t|^{-\beta_r}\norm{f}{L^{r'}},
  \qquad \beta_r=3\left(\frac12-\frac1r\right),
  \quad 2\leq r<\frac3\sigma.
  \label{eq:linear-attractive}
 \end{equation}
 In addition,
 \[
  \norm{U_a(t)f}{L^{4,1}}
  \lesssim_a |t|^{-3/4}\norm{f}{L^{4/3,1}}.
 \]
 \item If $-\frac14<a<0$, set
 $\sigma=\frac12-\sqrt{\frac14+a}$ and $p_*=3/\sigma$. Then
 \begin{equation}
  \norm{U_a(t)f}{L^{p_*,\infty}}
  \lesssim_{a,p_*}|t|^{-\beta_{p_*}}
  \norm{f}{L^{p_*',1}}.
  \label{eq:linear-weak-endpoint}
 \end{equation}
\end{enumerate}
\end{lemma}

\begin{proof}
For $a\geq0$, see
\cite[Theorem~1.11(i)]{FanelliDecay}; In the attractive case, estimate (2.2) is the three-dimensional specialization of \cite[Corollary~1.7]{MiaoSuZheng}, with its exponent $p$ replaced by $r'$. The $L^{4/3,1}\to L^{4,1}$ estimate follows by real interpolation between (2.2) at $r=3$ and $r=5$.

For the endpoint, set \(p_*=3/\sigma\), let \(P_0\) denote spherical averaging, and write \(Q=I-P_0\). We first consider \(f\in C_c^\infty(\mathbb R^3\setminus\{0\})\).

For the \(\ell=0\) component, the representation formula in \cite[Theorem~1.3 and Eqs.~(1.16), (1.27)]{FanelliDecay} gives
\[
|U_a(t)P_0f(x)|
\lesssim_a |t|^{-3/2}
\int_{\mathbb R^3}
\left|j_{-\sigma}\left(\frac{|x||y|}{2|t|}\right)\right|
|P_0f(y)|\,dy.
\]

Here
\[
j_{-\sigma}(r)=r^{-1/2}J_{\nu_0}(r),
\qquad
\nu_0=\sqrt{\frac14+a}=\frac12-\sigma.
\]

The standard Bessel estimates
\[
|J_{\nu_0}(r)|\lesssim_a
\begin{cases}
	r^{\nu_0},&0<r\le1,\\
	r^{-1/2},&r\ge1,
\end{cases}
\]

imply
\[
|j_{-\sigma}(r)|
\lesssim_a
\begin{cases}
	r^{-\sigma},&0<r\le1,\\
	r^{-1},&r\ge1,
\end{cases}
\lesssim_a r^{-\sigma},
\]

since \(0<\sigma<1/2\). Consequently,
\[
|U_a(t)P_0f(x)|
\lesssim_a
|t|^{-3/2+\sigma}|x|^{-\sigma}
\int_{\mathbb R^3}|y|^{-\sigma}|P_0f(y)|\,dy.
\]

Since \(|x|^{-\sigma}\in L^{p_*,\infty}\), Lorentz Hölder’s inequality and the boundedness of \(P_0\) on Lorentz spaces yield
\[
\|U_a(t)P_0f\|_{L^{p_*,\infty}}
\lesssim_a
|t|^{-3/2+\sigma}\|f\|_{L^{p_*',1}}
=
|t|^{-\beta_{p_*}}\|f\|_{L^{p_*',1}},
\]

where
\[
\beta_{p_*}=\frac32-\sigma
=3\left(\frac12-\frac1{p_*}\right).
\]

It remains to estimate the higher angular modes. In the notation of \cite[Eqs.~(6.5), (6.6)]{FanelliDecay},
\[
a_\ell=\frac12-\sqrt{\left(\ell+\frac12\right)^2+a}.
\]

Here \(a_0=\sigma>0\), whereas \(a_\ell<0\) for every \(\ell\ge1\). Thus the term \(S_2\) is precisely the kernel corresponding to \(U_a(t)Q\). Its uniform bound \cite[Eq.~(6.16)]{FanelliDecay} and the representation formula give
\[
\|U_a(t)Qf\|_{L^\infty}
\lesssim_a |t|^{-3/2}\|f\|_{L^1}.
\]

Moreover, \(Q\) is an orthogonal projection on \(L^2\), while \(U_a(t)\) is unitary on \(L^2\); hence
\[
\|U_a(t)Qf\|_{L^2}\le \|f\|_{L^2}.
\]

Real interpolation of these two estimates with parameter \(\theta=2/p_*\) and the second Lorentz index \(1\) gives
\[
U_a(t)Q:
(L^1,L^2)_{\theta,1}
\longrightarrow
(L^\infty,L^2)_{\theta,1},
\]

then,
\[
\|U_a(t)Qf\|_{L^{p_*,1}}
\lesssim_a |t|^{-\beta_{p_*}}
\|f\|_{L^{p_*',1}}.
\]

Using \(L^{p_*,1}\hookrightarrow L^{p_*,\infty}\) and combining the estimates for \(P_0f\) and \(Qf\), we conclude that
\[
\|U_a(t)f\|_{L^{p_*,\infty}}
\lesssim_a |t|^{-\beta_{p_*}}
\|f\|_{L^{p_*',1}}.
\]

The general case follows by density.
\end{proof}

\begin{lemma}[Strichartz estimates]
\label{lem:strichartz}
	Let \(a>-\frac14\), and let \(I\subset\mathbb R\) be an interval. For any Schrödinger-admissible pairs \((q,r)\) and \((\widetilde q,\widetilde r)\),
	
	$$
	\|U_a(t)f\|_{L_t^qL_x^r(\mathbb R\times\mathbb R^3)}
	\lesssim \|f\|_{L_x^2},
	$$
	
	and, for every \(t_0\in\mathbb R\),
	
	$$
	\left\|\int_{t_0}^t U_a(t-s)F(s)\,ds\right\|_{L_t^qL_x^r(I\times\mathbb R^3)}
	\lesssim
	\|F\|_{L_t^{\widetilde q'}L_x^{\widetilde r'}(I\times\mathbb R^3)} .
	$$
	
	If \(-\frac14<a<0\) and \((q,r)\) is nonendpoint admissible, then
	
	$$
	\|U_a(t)f\|_{L_t^{q,2}L_x^{r,2}}
	\lesssim_a \|f\|_{L_x^2}.
	$$
	
	Moreover, for \(2<\vartheta\le q\),
\begin{equation}
	\left\|\int_{t_0}^t U_a(t-s)F(s)\,ds\right\|_{
		L_t^{q,\vartheta}L_x^{r,2}(I\times\R^3)}
	\lesssim_{a,q,r,\vartheta}
	\norm{F}{L_t^2L_x^{6/5}(I\times\R^3)}.
	\label{eq:retarded-lorentz}
\end{equation}
	The implicit constants are independent of \(I\) and \(t_0\).
\end{lemma}

\begin{proof}
	The Lebesgue-space estimates, including the two endpoint, follow from
	\cite[Theorem~2.11 and Remark~2.13]{BoucletMizutani}; the hypotheses there hold for
	\(V(x)=a|x|^{-2}\) by Hardy's inequality.
	
	Assume $-\frac14<a<0$ and set $p_0=3/\sigma$. By
	\cite[Theorem~1.1]{MiaoSuZheng}, $W_\pm$ and $W_\pm^*$ are
	bounded on $L^p$ for $p_0'<p<p_0$. Since
	\[
	p_0'=\frac3{3-\sigma}<\frac65,\qquad p_0=\frac3\sigma>6,
	\]
	real interpolation gives their boundedness on $L^{s,2}$ for
	$6/5\le s\le6$.
	
	For every nonendpoint admissible pair $(q,r)$, Kowalski's estimate
	\cite[Proposition~2.8]{Kowalski} states that
	\[
	\|e^{it\Delta}g\|_{L_t^{q,2}L_x^{r,2}}
	\lesssim_{q,r}\|g\|_2,
	\qquad \frac2q+\frac3r=\frac32.
	\]
	Using $U_a(t)=W_\pm e^{it\Delta}W_\pm^*$ and the $L^2$-unitarity
	of $W_\pm^*$, we obtain
	\[
	\|U_a(t)f\|_{L_t^{q,2}L_x^{r,2}}
	\lesssim_{a,r}
	\|e^{it\Delta}W_\pm^*f\|_{L_t^{q,2}L_x^{r,2}}
	\lesssim_{q,r}\|f\|_2.
	\]
	
	To prove \eqref{eq:retarded-lorentz}, combine this estimate with the endpoint estimate and its dual:
	
	$$
	\begin{aligned}
		\left\|\int_{\mathbb R}U_a(t-s)F(s)\,ds\right\|_{L_t^{q,2}L_x^{r,2}}
		&\lesssim_a
		\left\|\int_{\mathbb R}U_a(-s)F(s)\,ds\right\|_{L_x^2}  \\
		&\lesssim_a \|F\|_{L_t^2L_x^{6/5}} .
	\end{aligned}
	$$
	
	Since \(L^{q,2}\hookrightarrow L^{q,\vartheta}\) for \(\vartheta\ge2\), Proposition~2.1 of \cite{AhnCho} converts the full integral into the time-ordered integral whenever \(2<\vartheta\le q\). Zero extension and time reversal give the estimate on any interval \(I\) and for any \(t_0\). A standard density argument completes the proof.
\end{proof}

\subsection{Adapted Sobolev and global spacetime estimates}

The following is the three-dimensional specialization of
\cite[Theorem~1.2]{KillipSobolev}. We state only the ranges used in
this paper.

\begin{lemma}[Adapted Sobolev estimates]
\label{lem:adapted-sobolev}
Let $a>-\frac14$, set
$\sigma=\frac12-\sqrt{\frac14+a}$, and let $0<s<2$. For
$f\in C_c^\infty(\R^3\setminus\{0\})$,
\[
 \norm{|\nabla|^sf}{L^r}
 \lesssim_{a,s,r}\norm{\cL_a^{s/2}f}{L^r}
\]
provided
\begin{equation}
 \frac{s+\sigma}{3}<\frac1r<
 \min\left\{1,\frac{3-\sigma}{3}\right\},
 \label{eq:sobolev-forward-range}
\end{equation}
whereas
\[
 \norm{\cL_a^{s/2}f}{L^r}
 \lesssim_{a,s,r}\norm{|\nabla|^sf}{L^r}
\]
provided
\[
 \max\left\{\frac{s}{3},\frac{\sigma}{3}\right\}<\frac1r<
 \min\left\{1,\frac{3-\sigma}{3}\right\}.
\]
If $1<r<q<\infty$, $1/q=1/r-s/3$, and
\eqref{eq:sobolev-forward-range} holds, then
\begin{equation}
 \norm{f}{L^{q,\theta}}
 \lesssim_{a,s,r,q,\theta}
 \norm{\cL_a^{s/2}f}{L^{r,\theta}},
 \qquad 1\leq\theta\leq\infty.
 \label{eq:adapted-embedding}
\end{equation}
\end{lemma}

\begin{lemma}[Global spacetime estimates]
\label{lem:global-spacetime}
Let $E=\norm{u_0}{\dot H^1}$.
\begin{enumerate}
 \item If $a\geq0$, then
 \begin{equation}
 \begin{split}
  &\norm{u}{L_t^4L_x^\infty}
  +\norm{u}{L_t^{20,10/3}L_x^{15/2}}\\
  &\quad
  +\norm{\cL_a^{1/2}u}{L_t^{8,2}L_x^{12/5,4}}
  +\norm{\cL_a^{1/2}u}{L_t^{8,4}L_x^{12/5,4}}
  \leq C_a(E).
 \end{split}
 \label{eq:nonnegative-spacetime}
 \end{equation}
 If also $u_0\in L^{3/2}$, then
 \[
  \norm{u(t)}{L^3}
  \leq C_a(E)|t|^{-1/2}\norm{u_0}{L^{3/2}},
  \qquad t\neq0.
 \]
 \item Suppose $-\frac14+\frac1{25}<a<0$. If $(q,r)$ is
 nonendpoint admissible, $q>2$, $2<r<6$, $2<\vartheta\leq q$, and
 $2\leq\varphi\leq\infty$, then
 \begin{equation}
  \norm{\cL_a^{1/2}u}{L_t^{q,\vartheta}L_x^{r,\varphi}}
  \leq C_{a,q,r,\vartheta,\varphi}(E).
  \label{eq:attractive-spacetime}
 \end{equation}
 The corresponding strong $L_t^qL_x^r$ estimate also holds.
\end{enumerate}
\end{lemma}

\begin{proof}
Part (1) follows from the square-function, persistence-of-regularity,
Lorentz--Strichartz estimates in
\cite[Lemmas~3.1, 3.4, and 3.6]{WangXuZhang} and \eqref{eq:adapted-embedding}. The last assertion is
\cite[Theorem~1.2]{WangXuZhang} with $p=3$.

For part (2), the local theory and the global bound
\eqref{eq:critical-bound} give, by the interval argument in
\cite[Proposition~2.10 and Theorem~2.11]{KillipNLS},
\[
 \norm{\cL_a^{1/2}u}{L_t^{10}L_x^{30/13}}
 +\norm{\cL_a^{1/2}(|u|^4u)}{L_t^2L_x^{6/5}}
 \leq C_a(E).
\]
Apply \eqref{eq:retarded-lorentz} to the differentiated Duhamel
formula. Spatial Lorentz nesting gives every $\varphi\geq2$, and time
nesting covers the stated range of $\vartheta$. The ordinary
Strichartz estimate gives the strong version.
\end{proof}

\section{Technical estimates}

The energy-class solution satisfies
\begin{equation}
 u(t)=U_a(t)u_0-i\int_0^tU_a(t-s)(|u|^4u)(s)\,ds
 \label{eq:duhamel}
\end{equation}
at every time \cite[Section~2]{KillipNLS}. Since
$|u|^4u\in C_tL_x^{6/5}\hookrightarrow C_t\Ham$, it is also an
identity of distributions.

The dispersive estimates below are compatible with
\eqref{eq:duhamel} by simultaneous smooth approximation. Cutting off
near zero and infinity and then mollifying gives density of
$C_c^\infty(\R^3\setminus\{0\})$ in $\dot H^1\cap L^{r,s}$,
$1\leq r,s<\infty$, with the sum norm; Hardy's inequality controls
the cutoff derivatives \cite[Chapter~2]{Ziemer}. The forcing is
approximated simultaneously in $L^{6/5}\hookrightarrow\Ham$ and the
relevant source space. Lemmas~\ref{lem:high-fractional},
\ref{lem:endpoint-fractional}, and
\ref{lem:nonnegative-endpoint-tools} justify the fractional and
endpoint extensions. Uniqueness of distributional limits identifies
these extensions with the energy Duhamel terms.

\subsection{Estimates for attractive potentials}

\begin{lemma}%[Auxiliary spacetime estimate]
\label{lem:auxiliary-below-six}
Assume $-\frac14+\frac1{25}<a<0$.  If $\frac6{3-4\sigma}<p<6,$ then
\begin{equation}
 \norm{u}{L_t^{8p/(6-p),4}L_x^{4p/(p-2)}}
 \leq C_{a,p}\bigl(\norm{u_0}{\dot H^1}\bigr).
 \label{eq:auxiliary-below-six}
\end{equation}
\end{lemma}

\begin{proof}
Write $E=\norm{u_0}{\dot H^1}$ and set
\[
 \widetilde q_p=\frac{8p}{6-p},
 \qquad \widetilde r_p=\frac{12p}{7p-6}.
\]
Then $(\widetilde q_p,\widetilde r_p)$ is admissible,
\[
\frac1{\widetilde r_p}-\frac13=\frac{p-2}{4p},
\qquad
p>\frac6{3-4\sigma}
\iff
\widetilde r_p<\frac3{1+\sigma}.
\]
Thus \eqref{eq:adapted-embedding}, with $s=1$, and
\eqref{eq:attractive-spacetime} yield
\[
\norm{u}{L_t^{\widetilde q_p,4}L_x^{4p/(p-2),4}}
\lesssim_{a,p}
\norm{\cL_a^{1/2}u}
{L_t^{\widetilde q_p,4}L_x^{\widetilde r_p,4}}
\leq C_{a,p}(E).
\]
Since $p<6$ implies $4<4p/(p-2)$, Lorentz nesting gives
$L_x^{4p/(p-2),4}\hookrightarrow L_x^{4p/(p-2)}$, which proves
\eqref{eq:auxiliary-below-six}.
\end{proof}

Let $S_t(f)$ denote the solution at time $t$ with initial datum $f$.
The next lemma provides the two endpoint Lipschitz bounds needed for
nonlinear interpolation.

\begin{lemma}[Difference estimates on an energy ball]
\label{lem:difference}
Let $-\frac14+\frac1{25}<a<0$ and
\[
 \norm{f}{\dot H^1}+\norm{g}{\dot H^1}\leq E.
\]
Then
\begin{align}
 \sup_{t\in\R}\norm{S_t(f)-S_t(g)}{L^2}
 &\leq C_{a,E}\norm{f-g}{L^2}, \label{eq:difference-L2}\\
 \norm{S_t(f)-S_t(g)}{L^{7/2}}
 &\leq C_{a,E}|t|^{-9/14}\norm{f-g}{L^{7/5}},
 \qquad t\neq0. \label{eq:difference-seven}
\end{align}
\end{lemma}

\begin{proof}
Write $u(t)=S_t(f)$, $v(t)=S_t(g)$, and $w=u-v$.  On a time interval $I$,
Strichartz and
\[
 \bigl||u|^4u-|v|^4v\bigr|
 \lesssim |w|(|u|^4+|v|^4)
\]
give
\[
 \begin{split}
 &\norm{w}{L_t^\infty L_x^2(I)}
 +\norm{w}{L_t^{10}L_x^{30/13}(I)}\\
 &\quad\lesssim_a \norm{w(t_I)}{L^2}
 +\bigl(\norm{u}{L_{t,x}^{10}(I)}^4+
         \norm{v}{L_{t,x}^{10}(I)}^4\bigr)
   \norm{w}{L_t^{10}L_x^{30/13}(I)}.
 \end{split}
\]
Theorem~\ref{thm:global} permits finite partitions of both time
half-lines on which the coefficient is small.  Absorption and finite
iteration prove \eqref{eq:difference-L2}.

For \eqref{eq:difference-seven}, observe that
$6/(3-4\sigma)<7/2$.  Lemma~\ref{lem:auxiliary-below-six} gives
\[
 \norm{u}{L_t^{56/5,4}L_x^{28/3}}
 +\norm{v}{L_t^{56/5,4}L_x^{28/3}}\leq C_{a,E}.
\]
The difference Duhamel formula and
\eqref{eq:linear-attractive} imply, for $t>0$,
\[
 \begin{split}
 \norm{w(t)}{L^{7/2}}
 &\lesssim_a t^{-9/14}\norm{f-g}{L^{7/5}}\\
 &\quad+\int_0^t(t-s)^{-9/14}\norm{w(s)}{L^{7/2}}
 \bigl(\norm{u(s)}{L^{28/3}}^4+
       \norm{v(s)}{L^{28/3}}^4\bigr)\,ds.
 \end{split}
\]
Here the fourth power of the coefficient lies in
$L_t^{14/5,1}$ and $t^{-9/14}\in L^{14/9,\infty}$.
Splitting at $t/2$, using Lorentz H\"older on the early part and
Young--O'Neil convolution on the late part, and then partitioning the
coefficient norm into finitely many small pieces gives
\[
 X_w(T_j)\leq C_{a,E}\norm{f-g}{L^{7/5}}
 +C_{a,E}X_w(T_{j-1})+C_a\eps^4X_w(T_j),
\]
where $X_w(T)=\sup_{0<t<T}t^{9/14}\norm{w(t)}{L^{7/2}}$.
For data for which $f-g\in L^2$, this supremum is finite on compact
intervals because $w\in C_t(L^2\cap L^6)\subset C_tL^{7/2}$.
Absorption and iteration prove the estimate in that case.

For general $f-g\in L^{7/5}$, apply the cutoff mollification approximation simultaneously to \(f\) and g.
\end{proof}

\begin{lemma}[Nonlinear interpolation below $7/2$]
\label{lem:nonlinear-interpolation}
Assume $-\frac14+\frac1{25}<a<0$.  If $2<p<7/2$ and
$u_0\in\dot H^1\cap L^{p'}$, then
\[
 \norm{S_t(u_0)}{L^p}
 \leq C_{a,p}\bigl(\norm{u_0}{\dot H^1}\bigr)
 |t|^{-3(\frac12-\frac1p)}\norm{u_0}{L^{p'}},
 \qquad t\neq0.
\]
\end{lemma}

\begin{proof}
	Set $E=\norm{u_0}{\dot H^1}$.  We use a localized version of the
	nonlinear $K$-functional argument in
	\cite[proof of Theorem~3, pp.~475--476]{Tartar}.  Recall that
	\[
	K(s,h;X_0,X_1)
	:=\inf_{h=h_0+h_1}
	\bigl(\norm{h_0}{X_0}+s\norm{h_1}{X_1}\bigr).
	\]
	For $\lambda>0$, define
	\[
	T_\lambda z=
	\begin{cases}
		z,& |z|\leq\lambda,\\
		\lambda z/|z|,& |z|>\lambda,
	\end{cases}
	\qquad R_\lambda=I-T_\lambda.
	\]
	These maps are Lipschitz on $\mathbb C\simeq\R^2$, vanish at zero,
	and satisfy
	\[
	\norm{T_\lambda u_0}{\dot H^1}
	+\norm{R_\lambda u_0}{\dot H^1}
	\lesssim \norm{u_0}{\dot H^1}
	\]
	by the Sobolev chain rule
	\cite[Theorem~2.1.11]{Ziemer}.  Moreover, since
	$7/5<p'<2$, one has $T_\lambda u_0\in L^2$ and
	$R_\lambda u_0\in L^{7/5}$.  The standard level-truncation formula
 \cite[Chapter~5]{BerghLofstrom} gives
	\[
	K(s,h;L^2,L^{7/5})
	\simeq
	\inf_{\lambda>0}
	\bigl(\norm{T_\lambda h}{L^2}
	+s\norm{R_\lambda h}{L^{7/5}}\bigr).
	\]
	
	Since $S_t(0)=0$, write
	\[
	S_t(u_0)
	=\bigl[S_t(u_0)-S_t(R_\lambda u_0)\bigr]
	+\bigl[S_t(R_\lambda u_0)-S_t(0)\bigr].
	\]
	All the initial data above remain in an energy ball depending only
	on $E$.  Hence \eqref{eq:difference-L2} and
	\eqref{eq:difference-seven} imply
	\[
	\begin{aligned}
		K(s,S_t(u_0);L^2,L^{7/2})
		&\lesssim_{a,E}
		\norm{T_\lambda u_0}{L^2}
		+s|t|^{-9/14}\norm{R_\lambda u_0}{L^{7/5}}.
	\end{aligned}
	\]
	Taking the infimum in $\lambda$ yields
	\[
	K(s,S_t(u_0);L^2,L^{7/2})
	\leq C_{a,E}
	K(C_{a,E}s|t|^{-9/14},u_0;L^2,L^{7/5}).
	\]
	
	Now set
	\[
	\theta=\frac{7(p-2)}{3p}\in(0,1).
	\]
	The real-interpolation identities give
	\[
	(L^2,L^{7/5})_{\theta,p'}=L^{p'},
	\qquad
	(L^2,L^{7/2})_{\theta,p'}=L^{p,p'}.
	\]
	Taking the interpolation norm and changing variables in the
	$K$-integral, we obtain
	\[
	\norm{S_t(u_0)}{L^{p,p'}}
	\leq C_{a,p}(E)|t|^{-9\theta/14}
	\norm{u_0}{L^{p'}}.
	\]
	Finally, $p'<p$ implies $L^{p,p'}\hookrightarrow L^p$, while
	\[
	\frac{9\theta}{14}
	=3\left(\frac12-\frac1p\right).
	\]
	This proves the result.
\end{proof}

Set $p_*=3/\sigma$. In the strong high-exponent range
$6\leq p<p_*$, write
\[
 \beta_p=3\left(\frac12-\frac1p\right),
 \qquad s_p=\frac34-\frac3p.
\]

\begin{lemma}[Fractional estimates below the attractive endpoint]
\label{lem:high-fractional}
Let $6\leq p<p_*$. Then
\begin{align}
 \norm{\cL_a^{-s_p/2}g}{L^p}
 &\lesssim_{a,p}\norm{g}{L^{4,1}},
 \label{eq:fractional-inverse}\\
 \norm{\cL_a^{s_p/2}(|u|^4u)}{L^{4/3,1}}
 &\lesssim_{a,p}
 \norm{u}{L^p}
 \norm{\cL_a^{1/2}u}{L^{24/11,4}}^4.
 \label{eq:fractional-nonlinear}
\end{align}
\end{lemma}

\begin{proof}
	Since $1/p=1/4-s_p/3$ and
	$(s_p+\sigma)/3<1/4$ is equivalent to $p<3/\sigma$,
	\eqref{eq:adapted-embedding} and Lorentz nesting give
	\[
	\norm{\cL_a^{-s_p/2}g}{L^p}
	\lesssim_{a,p}\norm{g}{L^{4,1}}.
	\]
	Set $1/\rho_p=3/8-1/p$.  By the reverse adapted Sobolev estimate and
	the Lorentz fractional product rule
	\cite[Lemma~3.8 and Corollary~3.9]{AnKimRyu},
	\[
	\norm{\cL_a^{s_p/2}(|u|^4u)}{L^{4/3,1}}
	\lesssim_{a,p}
	\norm{u}{L^p}\norm{u}{L^{8,4}}^3
	\norm{|\nabla|^{s_p}u}{L^{\rho_p,4}}.
	\]
	Here $L^p\hookrightarrow L^{p,\infty}$, while
	$3/4=1/p+3/8+1/\rho_p$ and $1=0+3/4+1/4$.  Moreover,
	\[
	\norm{u}{L^{8,4}}+
	\norm{|\nabla|^{s_p}u}{L^{\rho_p,4}}
	\lesssim_{a,p}
	\norm{\cL_a^{1/2}u}{L^{24/11,4}},
	\]
	by Lorentz Sobolev embedding and the adapted comparison; their range
	conditions follow from $\sigma<3/10<3/8$.  Combining the estimates
	proves the result for smooth functions, and the general case follows
	by density.
\end{proof}

At the limiting exponent, put
\[
 \beta_*=3\left(\frac12-\frac1{p_*}\right),
 \qquad s_*=\frac34-\frac3{p_*}.
\]

\begin{lemma}[Fractional estimates at the attractive endpoint]
\label{lem:endpoint-fractional}
For every $g\in L^{4,1}$,
\begin{equation}
 \norm{\cL_a^{-s_*/2}g}{L^{p_*,\infty}}
 \lesssim_a\norm{g}{L^{4,1}}.
 \label{eq:endpoint-fractional-inverse}
\end{equation}
If $u\in\dot H^1\cap L^{p_*,\infty}$ and
$\cL_a^{1/2}u\in L^{24/11,4}$, then
\begin{align}
 \norm{\cL_a^{s_*/2}(|u|^4u)}{L^{4/3,1}}
 &\lesssim_a
 \norm{u}{L^{p_*,\infty}}
 \norm{\cL_a^{1/2}u}{L^{24/11,4}}^4.
 \label{eq:endpoint-fractional-nonlinear}
\end{align}
\end{lemma}

Estimate \eqref{eq:endpoint-fractional-inverse} is a special case of 
\cite[Theorem~2.3]{LiuYuZhou}. We include a short direct proof for
completeness.

\begin{proof}
Since $s_*+\sigma=3/4$ and
$0<s_*<3-2\sigma$, the fractional-kernel estimate in
\cite[Lemma~2.2 and Eq.~(2.1)]{KillipSobolev} gives
\[
|\cL_a^{-s_*/2}g(x)|
\leq\cL_a^{-s_*/2}|g|(x)
\lesssim_a I_{s_*}|g|(x)
+|x|^{-\sigma}I_{3/4}|g|(x)
+I_{3/4}(|\,\cdot\,|^{-\sigma}|g|)(x),
\]
where $I_\alpha$ denotes the Euclidean Riesz potential. Young--O'Neil
fractional integration gives
\[
I_{s_*}:L^{4,1}\to L^{p_*,1},
\qquad I_{3/4}:L^{4,1}\to L^\infty.
\]
Moreover, $|x|^{-\sigma}\in L^{p_*,\infty}$ and, if
$1/r_0=1/4+1/p_*$, then Lorentz H\"older and fractional integration
give
\[
|\,\cdot\,|^{-\sigma}g\in L^{r_0,1},
\qquad I_{3/4}:L^{r_0,1}\to L^{p_*,1}.
\]
These three bounds prove \eqref{eq:endpoint-fractional-inverse}.
	
	For \eqref{eq:endpoint-fractional-nonlinear}, set
	$1/\rho_*=3/8-1/p_*$. For smooth $u$, the reverse adapted Sobolev
	comparison, the Lorentz fractional product rule
	\cite[Lemma~3.8 and Corollary~3.9]{AnKimRyu}, and Lorentz H\"older give
	\[
	\begin{aligned}
		\norm{\cL_a^{s_*/2}(|u|^4u)}{L^{4/3,1}}
		&\lesssim_a
		\norm{u}{L^{p_*,\infty}}\norm{u}{L^{8,4}}^3
		\norm{|\nabla|^{s_*}u}{L^{\rho_*,4}}\\
		&\lesssim_a
		\norm{u}{L^{p_*,\infty}}
		\norm{\cL_a^{1/2}u}{L^{24/11,4}}^4.
	\end{aligned}
	\]
	Here
	\[
	\frac34=\frac1{p_*}+\frac38+\frac1{\rho_*},
	\qquad
	\norm{u}{L^{8,4}}+
	\norm{|\nabla|^{s_*}u}{L^{\rho_*,4}}
	\lesssim_a\norm{\cL_a^{1/2}u}{L^{24/11,4}},
	\]
	and the required strict condition is
	$1/8-\sigma/3>0$, which follows from $\sigma<3/10$.
	
    Norm density in weak $L^{p_*}$ is not used.	For general $u$, 
smooth frequency truncation followed by spatial cutoffs and
	mollification gives $u_n\in C_c^\infty(\R^3\setminus\{0\})$ such that
	\[
	\sup_n\norm{u_n}{L^{p_*,\infty}}
	\lesssim\norm{u}{L^{p_*,\infty}},\qquad
	u_n\to u\ \text{in }\dot H^1\cap L^{8,4},\qquad
	|\nabla|^{s_*}u_n\to|\nabla|^{s_*}u
	\ \text{in }L^{\rho_*,4}.
	\]
	The inner cutoff is valid since $3/\rho_*-s_*=3/8>0$.
	Applying the preceding multilinear estimate to differences and using
	\[
	|u_n|^4u_n\to|u|^4u
	\quad\text{in }L^{6/5}\hookrightarrow\Ham
	\]
	allows us to pass to the limit and proves
	\eqref{eq:endpoint-fractional-nonlinear}.
\end{proof}

\begin{lemma}[Early-time estimate below the attractive endpoint]
\label{lem:early-high}
Let $6\leq p<p_*$, put
$E=\norm{u_0}{\dot H^1}$ and $A=\norm{u_0}{L^{p'}}$.
Then
\begin{equation}
 \int_0^\infty\norm{|u(s)|^4u(s)}{L^{p'}}\,ds
 \leq C_{a,p}(E)A.
 \label{eq:early-integrability}
\end{equation}
Consequently, for $t>0$,
\[
 \left\|\int_0^{t/2}U_a(t-s)(|u|^4u)(s)\,ds\right\|_{L^p}
 \leq C_{a,p}(E)t^{-\beta_p}A.
\]
\end{lemma}

\begin{proof}
The estimate at $p=7/2$, obtained by setting $g=0$ in
\eqref{eq:difference-seven}, is
\[
 \norm{u(s)}{L^{7/2}}
 \leq C_a(E)s^{-9/14}\norm{u_0}{L^{7/5}}.
\]
Define
\begin{align*}
 \alpha_p&=\frac{7(5p-6)}{23p},&
 k_p&=\frac{80p+42}{23p}=5-\alpha_p,&
 d_p&=\frac{9(5p-6)}{46p},\\
 R_p&=\frac{80p+42}{13p-11},&
 r_p&=\frac{240p+126}{119p+9},&
 q_p&=\frac{160p+84}{p+54},\\
 \eta_p&=\frac{80p+42}{13p+12}.&&&
\end{align*}
The identities needed below are
\begin{align*}
 \frac1{p'}&=\frac{\alpha_p}{7/2}+\frac{k_p}{R_p},&
 \frac1{R_p}&=\frac1{r_p}-\frac13,&
 \frac2{q_p}+\frac3{r_p}&=\frac32,\\
 d_p+\frac{k_p}{q_p}&=1,&
 \eta_p&=\frac{k_p}{1-2\alpha_p/7}.&&
\end{align*}
For $p\geq6$,
\[
 2<r_p<\frac{13}{6}<\frac3{1+\sigma},
 \qquad 2<k_p<q_p,\qquad \eta_p>2.
\]
Therefore Lemmas~\ref{lem:adapted-sobolev} and
\ref{lem:global-spacetime} give
\[
 \norm{u}{L_t^{q_p,k_p}L_x^{R_p,\eta_p}}\leq C_{a,p}(E).
\]
Interpolation between $L^{p'}$ and the energy embedding
$\dot H^1\hookrightarrow L^6$ yields
\[
 \norm{u_0}{L^{7/5}}^{\alpha_p}\leq C_{a,p}(E)A.
\]
The choice of $\eta_p$ and Lorentz H\"older in space now give
\[
 \norm{|u(s)|^4u(s)}{L^{p',1}}
 \leq C_{a,p}(E)A\,s^{-d_p}
 \norm{u(s)}{L^{R_p,\eta_p}}^{k_p}.
\]
Since $s^{-d_p}\in L^{1/d_p,\infty}$ and the last factor belongs to
$L_t^{q_p/k_p,1}$, the identity
$d_p+k_p/q_p=1$ proves \eqref{eq:early-integrability}.  The final
assertion follows from Lemma~\ref{lem:linear}, because
$|t-s|\simeq t$ on $(0,t/2)$.
\end{proof}

\begin{lemma}[Early-time estimate at the attractive endpoint]
\label{lem:endpoint-early-high}
Put $E=\norm{u_0}{\dot H^1}$ and
$A_*=\norm{u_0}{L^{p_*',1}}$. Then
\begin{equation}
 \int_0^\infty
 \norm{|u(s)|^4u(s)}{L^{p_*',1}}\,ds
 \leq C_a(E)A_*.
 \label{eq:endpoint-early-integrability}
\end{equation}
Consequently, for $t>0$,
\[
 \left\|\int_0^{t/2}U_a(t-s)(|u|^4u)(s)\,ds
 \right\|_{L^{p_*,\infty}}
 \leq C_a(E)t^{-\beta_*}A_*.
\]
\end{lemma}

\begin{proof}
Let $\alpha_*,k_*,d_*,R_*,r_*,q_*$, and $\eta_*$ be the values at
$p=p_*$ of the quantities in the proof of Lemma~\ref{lem:early-high}.
The same exponent identities hold, and $p_*>10$ gives
\[
 2<r_*<\frac{13}{6}<\frac3{1+\sigma},
 \qquad 2<k_*<q_*,\qquad \eta_*>2.
\]
Thus Lemmas~\ref{lem:adapted-sobolev} and
\ref{lem:global-spacetime} give
\[
 \norm{u}{L_t^{q_*,k_*}L_x^{R_*,\eta_*}}
 \leq C_a(E).
\]
The identity
\[
 \frac57=\frac{1/\alpha_*}{p_*'}
 +\frac{1-1/\alpha_*}{6}
\]
and Lorentz interpolation imply
$\norm{u_0}{L^{7/5}}^{\alpha_*}\leq C_a(E)A_*$.
Combining \eqref{eq:difference-seven}, with the second datum zero,
and spatial Lorentz H\"older yields
\[
 \norm{|u(s)|^4u(s)}{L^{p_*',1}}
 \leq C_a(E)A_*s^{-d_*}
 \norm{u(s)}{L^{R_*,\eta_*}}^{k_*}.
\]
Since $d_*+k_*/q_*=1$, Lorentz H\"older in time proves
\eqref{eq:endpoint-early-integrability}. The Duhamel estimate follows
from \eqref{eq:linear-weak-endpoint} and $|t-s|\simeq t$ on $(0,t/2)$.
\end{proof}

\begin{lemma}[Local control below the attractive endpoint]
\label{lem:local-high}
Let $6<p<p_*$. Then
\[
 \norm{u(t)}{L^p}\in L_{\mathrm{loc}}^{m_p}(\R_t),
 \qquad m_p=\frac{4p}{p-6}.
\]
For $p=6$, $u\in C_tL_x^6$.
\end{lemma}

\begin{proof}
	Let $6<p<p_*$ and set $r_0=3p/(p+3)$.  Since
	\[
	\frac2{m_p}+\frac3{r_0}=\frac32,
	\qquad
	\frac1p=\frac1{r_0}-\frac13,
	\]
	the global theory \cite[Theorem~1.2]{KillipNLS} and the Strichartz
	estimates \cite{BurqEtAl} give
	\[
	\norm{\cL_a^{1/2}u}{L_t^{m_p}L_x^{r_0}}\leq C_{a,p}(E).
	\]
	Moreover, $p<p_*=3/\sigma$ is precisely the condition needed to apply
	\eqref{eq:adapted-embedding} with $s=1$.  Therefore,
	\[
	\norm{u}{L_t^{m_p}L_x^p}
	\lesssim_{a,p}
	\norm{\cL_a^{1/2}u}{L_t^{m_p}L_x^{r_0}}
	\leq C_{a,p}(E).
	\]
	This proves the stronger global bound.  The case $p=6$ follows from
	$u\in C_t\dot H^1$ and $\dot H^1\hookrightarrow L^6$.
\end{proof}

\begin{lemma}[Local control at the attractive endpoint]
\label{lem:endpoint-local-high}
At $p_*=3/\sigma$,
\[
 \norm{u(t)}{L^{p_*,\infty}}
 \in L_{\mathrm{loc}}^{m_*}(\R_t),
 \qquad
 m_*:=\frac{4p_*}{p_*-6}=\frac4{1-2\sigma}>4.
\]
\end{lemma}

\begin{proof}
Put $\delta_*=1-s_*=\frac14+\sigma$ and
$r_{0,*}=3/(1+\sigma)$. Choose $(q_\pm,r_\pm)$, $R_\pm$, and
$\theta$ as in Lemma~\ref{lem:local-high}, with $p=p_*$. The
fractional Sobolev inequalities remain strict because
\[
 \delta_*+\sigma=\frac14+2\sigma<1+\sigma=\frac3{r_{0,*}}.
\]
Adapted fractional integration and real interpolation give, for almost every $t$,
\[
 \norm{\cL_a^{s_*/2}u(t)}{L^{4,1}}
 \lesssim_a
 \norm{\cL_a^{1/2}u(t)}{L^{r_-}}^{1-\theta}
 \norm{\cL_a^{1/2}u(t)}{L^{r_+}}^\theta.
\]
Apply \eqref{eq:endpoint-fractional-inverse}, the strong spacetime
bounds in Lemma~\ref{lem:global-spacetime}, and H\"older in time.
Admissibility yields
\[
 \frac{1-\theta}{q_-}+\frac{\theta}{q_+}
 =\frac14-\frac\sigma2
 =\frac1{m_*}.
\]
\end{proof}

\begin{proposition}[Sharpness of the strong attractive range]
\label{prop:sharpness}
Let $-\frac14<a<0$ and
$\sigma=\frac12-\sqrt{\frac14+a}$.  If $p\geq3/\sigma$, there is
$f\in\dot H^1(\R^3)\cap L^{p'}(\R^3)$ such that
$U_a(t)f\notin L^p(\R^3)$ for every $t\in\R$.
\end{proposition}

\begin{proof}
The singular harmonic oscillator eigenfunction $V_{0,1}$ in
\cite[Eq.~(1.11) and Theorem~2.4]{FanelliFrequency} is, up to a
constant,
\[
 f(x)=|x|^{-\sigma}e^{-|x|^2/4}.
\]
Because $\sigma<1/2$, $f\in\dot H^1\cap L^{p'}$.  Its explicit
evolution has the form
\[
 U_a(t)f(x)=c_a(t)|x|^{-\sigma}
 \exp\left(-\frac{|x|^2}{4(1+t^2)}\right)
 \exp\left(\frac{i|x|^2t}{4(1+t^2)}\right),
 \qquad c_a(t)\neq0.
\]
It is not locally in $L^p$ when $\sigma p\geq3$.
At $p=p_*$, the datum also belongs to $L^{p_*',1}$, whereas its
evolution is not locally in $L^{p_*,q}$ for any $q<\infty$. Thus the
weak target in \eqref{eq:attractive-endpoint} is sharp in its Lorentz
fine index.
\end{proof}

\subsection{Endpoint estimates for nonnegative potentials}

For $a\geq0$, define the endpoint adapted Sobolev space
\[
 \dot W_a^{1,(3/2,1)}
 :=\overline{C_c^\infty(\R^3\setminus\{0\})}^{\,
 \norm{\cL_a^{1/2}(\cdot)}{L^{3/2,1}}}.
\]

\begin{lemma}[Endpoint Sobolev--Lorentz estimates]
\label{lem:nonnegative-endpoint-tools}
Let $a\geq0$. For $t\neq0$, $U_a(t)$ extends uniquely to a bounded map
$\dot W_a^{1,(3/2,1)}\to L^\infty$ satisfying
\begin{equation}
 \norm{U_a(t)f}{L^\infty}
 \lesssim_a |t|^{-1/2}
 \norm{\cL_a^{1/2}f}{L^{3/2,1}}.
 \label{eq:hybrid-endpoint}
\end{equation}
Moreover, if $u\in\dot H^1\cap L^\infty$ and
$\cL_a^{1/2}u\in L^{12/5,4}$, then $|u|^4u$ belongs to
$\dot W_a^{1,(3/2,1)}$ and
\begin{equation}
 \norm{\cL_a^{1/2}(|u|^4u)}{L^{3/2,1}}
 \lesssim_a
 \norm{u}{L^\infty}
 \norm{\cL_a^{1/2}u}{L^{12/5,4}}^4.
 \label{eq:quintic-endpoint}
\end{equation}
\end{lemma}

\begin{proof}
Real interpolation of the $L^1\to L^\infty$ estimate
\eqref{eq:linear-nonnegative} with $L^2$ unitarity, at parameter
$2/3$ and second Lorentz exponent one, gives
\[
 \norm{U_a(t)g}{L^{3,1}}
 \lesssim_a |t|^{-1/2}\norm{g}{L^{3/2,1}}.
\]
For $a>0$, the adapted Riesz transform
$\nabla\cL_a^{-1/2}$ is bounded on $L^{3,1}$ by
Lemma~\ref{lem:adapted-sobolev}; for $a=0$, use the classical Riesz
transform theorem. Spectral commutation and the limiting
Sobolev--Lorentz inequality
\[
 \norm{h}{L^\infty}\lesssim\norm{\nabla h}{L^{3,1}},
\]
due to Alvino \cite{Alvino}, prove \eqref{eq:hybrid-endpoint} for
$f\in C_c^\infty(\R^3\setminus\{0\})$. Completion gives the asserted
unique extension. This is
the endpoint replacement for the ordinary Sobolev step in
\cite[Section~4.2]{WangXuZhang}.

For the nonlinear estimate, Lemma~\ref{lem:adapted-sobolev} and the
classical Lorentz Sobolev embedding give
\[
 \norm{\nabla u}{L^{12/5,4}}+\norm{u}{L^{12,4}}
 \lesssim_a\norm{\cL_a^{1/2}u}{L^{12/5,4}}.
\]
The Sobolev chain rule
\cite[Theorem~2.1.11]{Ziemer} and Lorentz H\"older therefore imply
\[
 \begin{split}
 \norm{\cL_a^{1/2}(|u|^4u)}{L^{3/2,1}}
 &\lesssim_a\norm{\nabla(|u|^4u)}{L^{3/2,1}}\\
 &\lesssim
 \norm{u}{L^\infty}\norm{u}{L^{12,4}}^3
 \norm{\nabla u}{L^{12/5,4}}\\
 & \lesssim_a
 \norm{u}{L^\infty}
 \norm{\cL_a^{1/2}u}{L^{12/5,4}}^4.
 \end{split}
\]
\end{proof}

\section{Proof of the main theorem}

Set $E=\norm{u_0}{\dot H^1}$. It suffices to consider $t>0$, since
$\overline{u(-t,x)}$ also solves \eqref{eq:nls}. All interval
iterations below are first performed on $(0,R)$, $R<\infty$, with
the last interval truncated at $R$; the bounds are uniform in $R$.

\subsection{Attractive potentials}

Assume throughout this subsection that
$-\frac14+\frac1{25}<a<0$.

\subsubsection{\texorpdfstring{The strong range $2<p<p_*$}%
 {The strong range 2 < p < p*}}

For $2<p<3/\sigma$, set $A=\norm{u_0}{L^{p'}}$ and
\[
 \beta_p=3\left(\frac12-\frac1p\right),
 \qquad
 X_p(T)=\operatorname*{ess\,sup}_{0<t<T}
 t^{\beta_p}\norm{u(t)}{L^p},\qquad X_p(0)=0.
\]

If $2<p<7/2$, estimate \eqref{eq:attractive-strong} is
Lemma~\ref{lem:nonlinear-interpolation}.  At $p=7/2$ it follows from
Lemma~\ref{lem:difference} with the second solution equal to zero.

For $7/2<p<6$, interpolation gives $u_0\in L^2$. The $L^2$
regularity in \cite[Proposition~2.10]{KillipNLS}, continued by
uniqueness, and the energy embedding imply
\[
 u\in C_t(L_x^2\cap L_x^6)\subset C_tL_x^p.
\]
Thus $X_p(T)$ is finite for every $T<\infty$.  By
Lemma~\ref{lem:auxiliary-below-six}, divide $[0,\infty)$ into finitely
many consecutive intervals $I_j=[T_{j-1},T_j)$ such that
\[
 \norm{u}{L_t^{8p/(6-p),4}L_x^{4p/(p-2)}(I_j\times\R^3)}
 <\eps.
\]
The number of intervals depends only on $a,p,E$, and $\eps$.

The linear term in \eqref{eq:duhamel} is bounded by
$C_{a,p}t^{-\beta_p}A$. For the nonlinear term, put
$F=|u|^4u$ and use
\[
 \norm{F}{L^{p'}}
 \leq\norm{u}{L^p}\norm{u}{L^{4p/(p-2)}}^4.
\]
Splitting its time integral at $t/2$, using $|t-s|\simeq t$ on the
first part and $s\simeq t$ on the second, and applying Lorentz
H\"older and Young--O'Neil convolution give
\[
 t^{\beta_p}
 \left\|\int_0^tU_a(t-s)F(s)\,ds\right\|_{L^p}
 \lesssim_{a,p}
 \left\|s^{\beta_p}\norm{u(s)}{L^p}
 \norm{u(s)}{L^{4p/(p-2)}}^4
 \right\|_{L_s^{2p/(6-p),1}(0,t)}.
\]
Here
\[
 s^{-\beta_p}\in L^{2p/[3(p-2)],\infty},
 \qquad
 \frac{3(p-2)}{2p}+\frac{6-p}{2p}=1.
\]
Splitting the time norm at $T_{j-1}$ and taking essential suprema gives
\[
 X_p(T_j)
 \leq C_{a,p}(E)A+C_{a,p}(E)X_p(T_{j-1})
 +C_{a,p}\eps^4X_p(T_j).
\]
Choose $\eps$ small, absorb the last term, and iterate. Letting
$R\to\infty$ proves \eqref{eq:attractive-strong} in this range.

For $6\leq p<p_*$, put
\[
 Z(t)=\norm{\cL_a^{1/2}u(t)}{L^{24/11,4}},
 \qquad
 M_p(t)=t^{\beta_p}\norm{u(t)}{L^p}.
\]
Lemma~\ref{lem:linear} and Lemma~\ref{lem:early-high} control the
linear and early nonlinear terms. For $t/2<s<t$, commutation with
the propagator, \eqref{eq:fractional-inverse}, the
$L^{4/3,1}\to L^{4,1}$ estimate in Lemma~\ref{lem:linear}, and
\eqref{eq:fractional-nonlinear} give
\[
 \norm{U_a(t-s)(|u|^4u)(s)}{L^p}
 \lesssim_{a,p}(t-s)^{-3/4}\norm{u(s)}{L^p}Z(s)^4.
\]
Since $s\simeq t$ on the late interval,
\[
 M_p(t)\leq C_{a,p}(E)A+
 C_{a,p}\int_{t/2}^t(t-s)^{-3/4}M_p(s)Z(s)^4\,ds.
\]
Here $Z\in L^{16,4}$ by \eqref{eq:attractive-spacetime}, and
$M_p\in L^m(0,R)$ by Lemma~\ref{lem:local-high}, with $m=m_p>4$
when $p>6$; take $m=8$ when $p=6$.
Lorentz interpolation gives
\[
 \norm{Z}{L^{16,4}(I)}
 \lesssim \norm{Z}{L^{16,8/3}(I)}^{3/5}
 \norm{Z}{L^{16}(I)}^{2/5}.
\]
Partitioning the integral of $Z^{16}$ therefore gives intervals
$I_j=[T_{j-1},T_j)$ with $\norm{Z}{L^{16,4}(I_j)}\leq\eps$ and
$J$ controlled by $a,E$, and $\eps$.

To justify finite suprema, split the scalar inequality at
$T_{j-1}$ and bound the earlier contribution and the constant term
by $B_{p,j}=C_{a,p}(E)(A+X_p(T_{j-1}))$. For small $\eps$,
Lorentz H\"older and Young--O'Neil \cite[Theorems~3.4--3.6]{ONeil}
show that $n$ substitutions on $I_j\cap(0,R)$ give constant terms
bounded by $2B_{p,j}$ and a remainder of $L^m$ norm at most
$2^{-n}\norm{M_p}{L^m(I_j\cap(0,R))}$. A subsequence tends to zero
almost everywhere; induction from $X_p(0)=0$ gives $X_p(R)<\infty$.
Taking essential suprema now yields the recurrence of
\cite[Section~3.2]{Kowalski} and \cite[Section~4.2]{WangXuZhang}:
\[
 X_p(T_j)
 \leq C_{a,p}(E)A+C_{a,p}(E)X_p(T_{j-1})
 +C_{a,p}\eps^4X_p(T_j).
\]
Absorption and finite iteration, followed by $R\to\infty$, yield
\[
 \operatorname*{ess\,sup}_{t>0}
 t^{\beta_p}\norm{u(t)}{L^p}
 \leq C_{a,p}(E)A.
\]

\subsubsection{\texorpdfstring{The weak endpoint $p=p_*$}%
 {The weak endpoint p = p*}}

Set
\[
 p_*=\frac3\sigma,\qquad
 \beta_*=\frac32-\sigma,\qquad
 A_*=\norm{u_0}{L^{p_*',1}},
\]
and define
\[
 M_*(t)=t^{\beta_*}\norm{u(t)}{L^{p_*,\infty}},
 \qquad X_*(T)=\operatorname*{ess\,sup}_{0<t<T}M_*(t),
 \qquad X_*(0)=0.
\]
The weak linear endpoint \eqref{eq:linear-weak-endpoint} gives
\[
 \norm{U_a(t)u_0}{L^{p_*,\infty}}
 \lesssim_a t^{-\beta_*}A_*,
\]
and Lemma~\ref{lem:endpoint-early-high} gives the same bound for the
Duhamel integral over $(0,t/2)$. For the late integral, commute
$\cL_a^{s_*/2}$ with the propagator and apply, in order,
\eqref{eq:endpoint-fractional-inverse}, the
$L^{4/3,1}\to L^{4,1}$ linear estimate, and
\eqref{eq:endpoint-fractional-nonlinear}. Thus
\[
 \norm{U_a(t-s)(|u|^4u)(s)}{L^{p_*,\infty}}
 \lesssim_a(t-s)^{-3/4}
 \norm{u(s)}{L^{p_*,\infty}}Z(s)^4.
\]
It follows that, for almost every $t>0$,
\[
 M_*(t)\leq C_a(E)A_*+
 C_a\int_{t/2}^t(t-s)^{-3/4}M_*(s)Z(s)^4\,ds.
\]
Here $Z$ is as above. On the same small intervals,
Lemma~\ref{lem:endpoint-local-high} gives $M_*\in L^{m_*}(0,R)$,
$m_*>4$, so the preceding substitution argument gives
$X_*(R)<\infty$. Lorentz H\"older and essential suprema then yield
\[
 X_*(T_j)
 \leq C_a(E)A_*+C_a(E)X_*(T_{j-1})
 +C_a\eps^4X_*(T_j).
\]
Absorbing the last term and iterating gives, as $R\to\infty$,
\[
 \operatorname*{ess\,sup}_{t>0}
 t^{\beta_*}\norm{u(t)}{L^{p_*,\infty}}
 \leq C_a(E)A_*.
\]

\subsection{The nonnegative endpoint}

Assume $a\geq0$ and set $A=\norm{u_0}{L^1}$.
The linear term satisfies
\[
 \norm{U_a(t)u_0}{L^\infty}\lesssim_a t^{-3/2}A.
\]
Interpolation gives $\norm{u_0}{L^{3/2}}\lesssim A^{3/5}E^{2/5}$;
hence \cite[Theorem~1.2, $p=3$]{WangXuZhang} implies
\[
 \norm{u(s)}{L^3}^{5/3}
 \leq C_a(E)s^{-5/6}A.
\]
Since $\norm{|u|^4u}{L^1}\leq
\norm{u}{L^3}^{5/3}\norm{u}{L^{15/2}}^{10/3}$ and
$\norm{u}{L_x^{15/2}}^{10/3}\in L_t^{6,1}$ by
\eqref{eq:nonnegative-spacetime}, Lorentz H\"older against
$s^{-5/6}\in L^{6/5,\infty}$ gives
\[
 \int_0^\infty\norm{|u(s)|^4u(s)}{L^1}\,ds
 \leq C_a(E)A.
\]
As $|t-s|\simeq t$ for $0<s<t/2$,
\begin{equation}
 \left\|\int_0^{t/2}U_a(t-s)(|u|^4u)(s)\,ds\right\|_{L^\infty}
 \leq C_a(E)t^{-3/2}A.
 \label{eq:early-nonnegative}
\end{equation}

Put
\[
 G(t)=\norm{\cL_a^{1/2}u(t)}{L^{12/5,4}}
\]
and, for $0<T\leq\infty$, define
\[
 \norm{u}{X(T)}
 :=\operatorname*{ess\,sup}_{0<t<T}
       t^{3/2}\norm{u(t)}{L^\infty},
 \qquad \norm{u}{X(0)}:=0.
\]
By \eqref{eq:hybrid-endpoint}, \eqref{eq:quintic-endpoint}, and
\eqref{eq:early-nonnegative}, using $s\simeq t$ on the late interval,
\[
 t^{3/2}\norm{u(t)}{L^\infty}
 \leq C_a(E)A+C_a\int_{t/2}^t(t-s)^{-1/2}
 s^{3/2}\norm{u(s)}{L^\infty}G(s)^4\,ds.
\]

As in the attractive case, the bounds $G\in L^{8,2}\cap L^{8,4}$
from Lemma~\ref{lem:global-spacetime} give a finite partition
$I_j=[T_{j-1},T_j)$ such that
\[
 \norm{G}{L^{8,4}(I_j)}\leq\eps,
 \qquad 1\leq j\leq J,
\]
with $J$ controlled by $a,E$, and $\eps$. Since
$t^{3/2}\norm{u(t)}{L^\infty}\in L^4(0,R)$ by
\eqref{eq:nonnegative-spacetime}, the preceding substitution
argument gives $\norm{u}{X(R)}<\infty$, now using the kernel
$t^{-1/2}$ and $G^4\in L^{2,1}$. Splitting at $T_{j-1}$, applying
Lorentz H\"older, and taking essential suprema gives
\[
 \norm{u}{X(T_j)}
 \leq C_a(E)A+C_a(E)\norm{u}{X(T_{j-1})}
 +C_a\eps^4\norm{u}{X(T_j)}.
\]
Absorption and finite iteration give, as $R\to\infty$,
\[
 \operatorname*{ess\,sup}_{t>0}
 t^{3/2}\norm{u(t)}{L^\infty}\leq C_a(E)A.
\]

To extend the three estimates to every $t_0>0$, choose
$t_n\to t_0$ outside the exceptional set.  Energy continuity gives
$u(t_n)\to u(t_0)$ in distributions.  By reflexivity, or weak-$*$
compactness at the two endpoints, a subsequence converges in the
corresponding target space.  Its limit must be $u(t_0)$; hence weak lower
semicontinuity gives the desired estimate at $t_0$.  Negative times follow
by time reversal.

\setlength{\bibsep}{6pt plus 1pt minus 1pt}


\begin{thebibliography}{99}

\bibitem{AhnCho}
C. Ahn, Y. Cho,
Lorentz space extension of Strichartz estimates,
\emph{Proc. Am. Math. Soc.} 133 (2005) 3497--3503.
\href{https://doi.org/10.1090/S0002-9939-05-07891-3}
{doi:10.1090/S0002-9939-05-07891-3}.

\bibitem{Alvino}
A. Alvino,
Sulla diseguaglianza di Sobolev in spazi di Lorentz,
\emph{Boll. Un. Mat. Ital. A (5)} 14 (1977) 148--156.

\bibitem{AnKimRyu}
J. An, J. Kim, P. Ryu,
Sobolev--Lorentz spaces with an application to the inhomogeneous
biharmonic NLS equation,
\emph{Discrete Contin. Dyn. Syst. Ser. B} 29 (2024) 3326--3345.
\href{https://doi.org/10.3934/dcdsb.2024006}
{doi:10.3934/dcdsb.2024006}.

\bibitem{BennettSharpley}
C. Bennett, R. Sharpley,
\emph{Interpolation of Operators},
Pure and Applied Mathematics, vol.~129,
Academic Press, Boston, 1988.

\bibitem{BerghLofstrom}
J. Bergh, J. L\"ofstr\"om,
\emph{Interpolation Spaces: An Introduction},
Grundlehren der mathematischen Wissenschaften, vol.~223,
Springer-Verlag, Berlin--New York, 1976.
\href{https://doi.org/10.1007/978-3-642-66451-9}
{doi:10.1007/978-3-642-66451-9}.

\bibitem{BoucletMizutani}
J.-M. Bouclet, H. Mizutani,
Uniform resolvent and Strichartz estimates for Schr\"odinger
equations with critical singularities,
\emph{Trans. Am. Math. Soc.} 370 (2018) 7293--7333.
\href{https://doi.org/10.1090/tran/7243}
{doi:10.1090/tran/7243}.

\bibitem{BurqEtAl}
N. Burq, F. Planchon, J.~G. Stalker, A.~S. Tahvildar-Zadeh,
Strichartz estimates for the wave and Schr\"odinger equations with the
inverse-square potential,
\emph{J. Funct. Anal.} 203 (2003) 519--549.
\href{https://doi.org/10.1016/S0022-1236(03)00238-6}
{doi:10.1016/S0022-1236(03)00238-6}.

\bibitem{FanKillipVisanZhao}
C. Fan, R. Killip, M. Visan, Z. Zhao,
Dispersive decay for the mass-critical nonlinear Schr\"odinger
equation,
\emph{Math. Z.} 311 (2025) Art.~21, 16 pp.
\href{https://doi.org/10.1007/s00209-025-03821-8}
{doi:10.1007/s00209-025-03821-8}.

\bibitem{FanStaffilaniZhao}
C. Fan, G. Staffilani, Z. Zhao,
On decaying properties of nonlinear Schr\"odinger equations,
\emph{SIAM J. Math. Anal.} 56 (2024) 3082--3109.
\href{https://doi.org/10.1137/23M1557544}
{doi:10.1137/23M1557544}.

\bibitem{FanZhao2021}
C. Fan, Z. Zhao,
Decay estimates for nonlinear Schr\"odinger equations,
\emph{Discrete Contin. Dyn. Syst.} 41 (2021) 3973--3984.
\href{https://doi.org/10.3934/dcds.2021024}
{doi:10.3934/dcds.2021024}.

\bibitem{FanZhao2023}
C. Fan, Z. Zhao,
A note on decay property of nonlinear Schr\"odinger equations,
\emph{Proc. Am. Math. Soc.} 151 (2023) 2527--2542.
\href{https://doi.org/10.1090/proc/16296}
{doi:10.1090/proc/16296}.

\bibitem{FanelliDecay}
L. Fanelli, V. Felli, M.~A. Fontelos, A. Primo,
Time decay of scaling critical electromagnetic Schr\"odinger flows,
\emph{Commun. Math. Phys.} 324 (2013) 1033--1067.
\href{https://doi.org/10.1007/s00220-013-1830-y}
{doi:10.1007/s00220-013-1830-y}.

\bibitem{FanelliFrequency}
L. Fanelli, V. Felli, M.~A. Fontelos, A. Primo,
Frequency-dependent time decay of Schr\"odinger flows,
\emph{J. Spectr. Theory} 8 (2018) 509--521.
\href{https://doi.org/10.4171/JST/204}
{doi:10.4171/JST/204}.

\bibitem{GuoHuangSong}
Z. Guo, C. Huang, L. Song,
Pointwise decay of solutions to the energy critical nonlinear
Schr\"odinger equations,
\emph{J. Differ. Equ.} 366 (2023) 71--84.
\href{https://doi.org/10.1016/j.jde.2023.04.018}
{doi:10.1016/j.jde.2023.04.018}.

\bibitem{Hunt}
R.~A. Hunt,
An extension of the Marcinkiewicz interpolation theorem to Lorentz
spaces,
\emph{Bull. Am. Math. Soc.} 70 (1964) 803--807.
\href{https://doi.org/10.1090/S0002-9904-1964-11242-8}
{doi:10.1090/S0002-9904-1964-11242-8}.

\bibitem{KeelTao}
M. Keel, T. Tao,
Endpoint Strichartz estimates,
\emph{Am. J. Math.} 120 (1998) 955--980.
\href{https://doi.org/10.1353/ajm.1998.0039}
{doi:10.1353/ajm.1998.0039}.

\bibitem{KillipNLS}
R. Killip, C. Miao, M. Visan, J. Zhang, J. Zheng,
The energy-critical NLS with inverse-square potential,
\emph{Discrete Contin. Dyn. Syst.} 37 (2017) 3831--3866.
\href{https://doi.org/10.3934/dcds.2017162}
{doi:10.3934/dcds.2017162}.

\bibitem{KillipSobolev}
R. Killip, C. Miao, M. Visan, J. Zhang, J. Zheng,
Sobolev spaces adapted to the Schr\"odinger operator with inverse-square
potential,
\emph{Math. Z.} 288 (2018) 1273--1298.
\href{https://doi.org/10.1007/s00209-017-1934-8}
{doi:10.1007/s00209-017-1934-8}.

\bibitem{Kowalski}
M. Kowalski,
Dispersive decay for the energy-critical nonlinear Schr\"odinger
equation,
\emph{J. Differ. Equ.} 429 (2025) 392--426.
\href{https://doi.org/10.1016/j.jde.2025.02.040}
{doi:10.1016/j.jde.2025.02.040}.

\bibitem{LiuYuZhou}
H. Liu, Q. Yu, H. Zhou,
Sharp and endpoint two-weight fractional integral estimates for
Schr\"odinger operators with inverse-square potentials,
\emph{arXiv preprint} arXiv:2607.09585 (2026).
\href{https://doi.org/10.48550/arXiv.2607.09585}
{doi:10.48550/arXiv.2607.09585}.


\bibitem{MiaoSuZheng}
C. Miao, X. Su, J. Zheng,
The $W^{s,p}$-boundedness of stationary wave operators for the
Schr\"odinger operator with the inverse-square potential,
\emph{Trans. Am. Math. Soc.} 376 (2023) 1739--1797.
\href{https://doi.org/10.1090/tran/8823}
{doi:10.1090/tran/8823}.

\bibitem{ONeil}
R. O'Neil,
Convolution operators and $L(p,q)$ spaces,
\emph{Duke Math. J.} 30 (1963) 129--142.
\href{https://doi.org/10.1215/S0012-7094-63-03015-1}
{doi:10.1215/S0012-7094-63-03015-1}.


\bibitem{Tartar}
L. Tartar,
Interpolation non lin\'eaire et r\'egularit\'e,
\emph{J. Funct. Anal.} 9 (1972) 469--489.
\href{https://doi.org/10.1016/0022-1236(72)90022-5}
{doi:10.1016/0022-1236(72)90022-5}.

\bibitem{WangXuZhang}
J. Wang, C. Xu, F. Zhang,
Decay estimates for nonlinear Schr\"odinger equation with the
inverse-square potential,
\emph{J. Math. Anal. Appl.} 550 (2025) Art.~129631.
\href{https://doi.org/10.1016/j.jmaa.2025.129631}
{doi:10.1016/j.jmaa.2025.129631}.

\bibitem{Ziemer}
W.~P. Ziemer,
\emph{Weakly Differentiable Functions: Sobolev Spaces and Functions of
Bounded Variation},
Graduate Texts in Mathematics, vol.~120,
Springer, New York, 1989.
\href{https://doi.org/10.1007/978-1-4612-1015-3}
{doi:10.1007/978-1-4612-1015-3}.

\end{thebibliography}
\end{document}